\documentclass[preprint,12pt]{elsarticle}

\usepackage[T1]{fontenc}
\usepackage[utf8]{inputenc}
\usepackage{amsmath,amssymb,amsthm,mathtools}
\usepackage{enumitem}
\usepackage{microtype}
\usepackage[colorlinks=true,linkcolor=blue,citecolor=blue,urlcolor=blue]{hyperref}
\hypersetup{%
  pdftitle={Diameter-Ramsey triangles below the 135-degree threshold},%
  pdfauthor={Yaping Mao}%
}

\theoremstyle{plain}
\newtheorem{theorem}{Theorem}[section]
\newtheorem{lemma}[theorem]{Lemma}
\newtheorem{corollary}[theorem]{Corollary}
\theoremstyle{definition}
\newtheorem{definition}[theorem]{Definition}
\theoremstyle{remark}
\newtheorem{remark}[theorem]{Remark}

\DeclareMathOperator{\diam}{diam}
\DeclareMathOperator{\crad}{cr}
\newcommand{\R}{\mathbb{R}}

\newcommand{\calT}{\mathcal{T}}

\begin{document}

\begin{frontmatter}

\title{Diameter-Ramsey triangles below the \texorpdfstring{$135^\circ$}{135-degree} threshold}

\author[aff1]{Yaping Mao\corref{cor1}\fnref{fn1}}
\cortext[cor1]{Corresponding author.}
\fntext[fn1]{Supported by the National Science Foundation of China (Nos.\ 12471329 and 12061059).}
\ead{yapingmao@outlook.com; myp@qhnu.edu.cn}

\address[aff1]{Academy of Plateau Science and Sustainability, and School of Mathematics and Statistics, Qinghai Normal University, Xining, Qinghai 810008, China}

\begin{abstract}
A finite Euclidean set is diameter-Ramsey if, for every number of colors, some finite same-diameter witness has the property that every coloring of the witness contains a monochromatic congruent copy of the set. Frankl, Pach, Reiher and R\"odl asked whether any obtuse triangle is diameter-Ramsey. We prove the stronger statement that every non-degenerate triangle whose largest angle is strictly smaller than $135^\circ$ is diameter-Ramsey. Together with the theorem of Corsten and Frankl that triangles with an angle larger than $135^\circ$ are not diameter-Ramsey, this gives the sharp classification for the two open angular ranges on either side of $135^\circ$. The proof uses a weighted $k$-subset configuration with non-negative coefficients; a finite binary-tree construction realizes the required two prescribed overlaps, and the ordinary hypergraph Ramsey theorem then forces a monochromatic copy of the triangle.
\end{abstract}


\end{frontmatter}

\section{Introduction}

Euclidean Ramsey theory, initiated by Erd\H{o}s, Graham, Montgomery, Rothschild, Spencer and Straus \cite{ErdosEtAl1973}, asks which finite Euclidean configurations are forced to occur monochromatically under arbitrary finite colorings of sufficiently rich Euclidean sets or spaces.  In the finite-witness form used here, a finite set $P$ in a Euclidean space is called \emph{Ramsey} if, for every integer $r\ge2$, there is a finite Euclidean set $R$ such that every $r$-coloring of the points of $R$ contains a monochromatic subset congruent to $P$.  A classical theorem of Frankl and R\"odl states that every triangle is Ramsey \cite{FranklRodl1986}; later they proved more generally that every non-degenerate simplex is Ramsey \cite{FranklRodl1990}.

Frankl, Pach, Reiher and R\"odl introduced the following same-diameter strengthening in their work on Borsuk and Ramsey type questions \cite{FranklPachReiherRodl2018}. For a finite Euclidean set $P$, write
\[
   \diam(P)=\max\{\|x-y\|:x,y\in P\}.
\]

\begin{definition}[Diameter-Ramsey set]\label{def:diameter-ramsey}
A finite set $P$ in a Euclidean space is \emph{diameter-Ramsey} if, for every integer $r\ge2$, there exist an integer $D$ and a finite set $R\subset\R^D$ such that
\[
   \diam(R)=\diam(P)
\]
and every $r$-coloring of the points of $R$ contains a monochromatic subset congruent to $P$.
\end{definition}

Every diameter-Ramsey set is Ramsey, but the converse is false.  For triangles, Frankl, Pach, Reiher and R\"odl proved that all acute and right triangles are diameter-Ramsey, while no triangle with an angle larger than $150^\circ$ is diameter-Ramsey \cite{FranklPachReiherRodl2018}.  They asked whether there exists an obtuse triangle that is diameter-Ramsey.  Corsten and Frankl subsequently proved that every finite set of diameter $1$ whose circumradius is strictly larger than $1/\sqrt2$ is not diameter-Ramsey; in particular, every triangle with an angle larger than $135^\circ$ is not diameter-Ramsey \cite{CorstenFrankl2018}.

The purpose of this note is to close the positive side up to the same threshold.

\begin{theorem}\label{thm:main}
Let $T$ be a non-degenerate triangle.  If the largest angle of $T$ is strictly smaller than $135^\circ$, then $T$ is diameter-Ramsey.
\end{theorem}

Thus the original question has a strong affirmative answer: not only do obtuse diameter-Ramsey triangles exist, but all obtuse triangles with largest angle below $135^\circ$ are diameter-Ramsey.  Combining Theorem~\ref{thm:main} with the obstruction of Corsten and Frankl gives a sharp open-interval threshold at $135^\circ$; the exact endpoint is not addressed here.

We also obtain an immediate consequence for a conjecture of Corsten and Frankl.  They conjectured that a simplex is diameter-Ramsey if and only if its circumcenter belongs to its convex hull \cite[Conjecture 3.1]{CorstenFrankl2018}.  Since the circumcenter of an obtuse triangle lies outside the triangle, Theorem~\ref{thm:main} refutes this conjecture already for two-dimensional simplices.

The proof is elementary apart from the finite hypergraph Ramsey theorem.  The main construction is a family of weighted $k$-subset vectors
\[
   x_I=\sum_{j=1}^k\alpha_j e_{i_j},\qquad I=\{i_1<\cdots<i_k\},
\]
where all coefficients are non-negative.  Non-negativity keeps the diameter under control, while a finite binary-tree pattern realizes two prescribed overlap inner products.  A Ramsey argument then makes the three desired $k$-subsets monochromatic.

\section{Weighted subset configurations}

We write $[n]=\{1,\ldots,n\}$ and, for a finite set $X$, we write $\binom{X}{k}$ for the family of all $k$-element subsets of $X$.  For integers $k,m,r\ge1$, let $R_k(m;r)$ denote a $k$-uniform hypergraph Ramsey number: every $r$-coloring of $\binom{[N]}{k}$ with $N\ge R_k(m;r)$ contains an $m$-element set $S$ such that all members of $\binom{S}{k}$ have the same color.

Let $\Omega$ be a finite linearly ordered set and let
\[
   U=\{u_1<\cdots<u_k\},\qquad V=\{v_1<\cdots<v_k\}
\]
be $k$-subsets of $\Omega$.  For a coefficient vector $\alpha=(\alpha_1,\ldots,\alpha_k)$, define the \emph{ordered overlap form}
\begin{equation}\label{eq:mu}
   \mu_\alpha(U,V)
   =\sum_{\substack{1\le s,t\le k\\ u_s=v_t}}\alpha_s\alpha_t.
\end{equation}
This notation records which coefficient position a common element occupies in each ordered subset.

\begin{lemma}[Weighted $k$-subset model]\label{lem:model}
Let $\alpha_1,\ldots,\alpha_k\ge0$ and put
\[
   L=\sum_{j=1}^k\alpha_j^2>0.
\]
For $n\ge k$ and $I=\{i_1<\cdots<i_k\}\in\binom{[n]}{k}$, define
\[
   x_I=\sum_{j=1}^k\alpha_j e_{i_j}\in\R^n,
\]
where $e_1,\ldots,e_n$ are the standard unit vectors.  Then
\begin{equation}\label{eq:inner-product-model}
   \langle x_I,x_J\rangle=\mu_\alpha(I,J)
\end{equation}
for all $I,J\in\binom{[n]}{k}$, where $[n]$ has its natural order.  In particular, if $n\ge2k$, then
\[
   X_\alpha(n):=\{x_I:I\in\binom{[n]}{k}\}
\]
has diameter $\sqrt{2L}$.
\end{lemma}

\begin{proof}
Equation \eqref{eq:inner-product-model} follows by expanding the inner product and using the orthonormality of the standard basis.  Each vector $x_I$ has squared norm $L$.  Since the coefficients are non-negative, all inner products are non-negative; hence
\[
   \|x_I-x_J\|^2=2L-2\langle x_I,x_J\rangle\le 2L.
\]
If $n\ge2k$, there are disjoint $k$-subsets $I,J\subset[n]$.  For them $\langle x_I,x_J\rangle=0$, so the upper bound $2L$ is attained.  Thus $\diam(X_\alpha(n))=\sqrt{2L}$.
\end{proof}

The next lemma is the main realization step.  It says that every point strictly inside a quarter disk can be represented by two ordered overlaps with one common coefficient vector.

\begin{lemma}[Finite binary-tree representation]\label{lem:tree}
Let $L>0$ and let $p,q\ge0$ satisfy
\[
   p^2+q^2<L^2.
\]
Then there exist an integer $k\ge1$, non-negative coefficients $\alpha_1,\ldots,\alpha_k$ satisfying
\[
   \sum_{j=1}^k\alpha_j^2=L,
\]
a finite linearly ordered set $\Omega$, and three $k$-subsets $A,B,C\in\binom{\Omega}{k}$ such that
\[
   \mu_\alpha(A,B)=0,
   \qquad
   \mu_\alpha(A,C)=p,
   \qquad
   \mu_\alpha(B,C)=q.
\]
\end{lemma}

\begin{proof}
If $p=q=0$, take $k=1$, $\alpha_1=\sqrt L$, and take $A,B,C$ to be three distinct singleton subsets of a three-element ordered set.  Thus assume $p^2+q^2>0$.

Set
\[
   \rho=\frac{\sqrt{p^2+q^2}}{L}\in(0,1).
\]
Choose an integer $M\ge1$ such that
\[
   \rho<\frac{M}{M+1}.
\]
For $0\le \eta<1$ define
\[
   f_M(\eta)=\sqrt{\eta}\,\frac{1-\eta^M}{1-\eta^{M+1}}.
\]
The function $f_M$ extends continuously to $\eta=1$, with $f_M(1)=M/(M+1)$.  Since $f_M(0)=0$ and $\rho<M/(M+1)$, the intermediate value theorem gives a number $\eta\in(0,1)$ such that
\begin{equation}\label{eq:fM-choice}
   \sqrt{\eta}\,\frac{1-\eta^M}{1-\eta^{M+1}}=\rho.
\end{equation}
Define
\[
   s=\frac{p}{\sqrt{p^2+q^2}}\sqrt{\eta},
   \qquad
   t=\frac{q}{\sqrt{p^2+q^2}}\sqrt{\eta}.
\]
Then $s,t\ge0$ and $s^2+t^2=\eta$.

Let $\calT_M$ be the set of all binary words of length at most $M$, including the empty word $\varnothing$, ordered first by length and then lexicographically with $0<1$.  We define coefficients indexed by $\calT_M$ recursively.  Put
\[
   \alpha_\varnothing=\lambda,
   \qquad
   \alpha_{w0}=s\alpha_w,
   \qquad
   \alpha_{w1}=t\alpha_w,
\]
where
\[
   \lambda^2=L\frac{1-\eta}{1-\eta^{M+1}}.
\]
For each level $j$,
\[
   \sum_{|w|=j}\alpha_w^2=\lambda^2(s^2+t^2)^j=\lambda^2\eta^j.
\]
Hence
\[
   \sum_{w\in\calT_M}\alpha_w^2
   =\lambda^2\sum_{j=0}^M\eta^j=L.
\]
Let $k=|\calT_M|=2^{M+1}-1$, and list the coefficients $\alpha_w$ in the order of $\calT_M$ to obtain $(\alpha_1,\ldots,\alpha_k)$.

It remains to define the ordered pattern.  Introduce symbols
\[
   z,
   \qquad
   a_w,b_w\quad(w\in\calT_M).
\]
Set
\[
   A=\{a_w:w\in\calT_M\},
   \qquad
   B=\{b_w:w\in\calT_M\},
\]
and
\[
   C=\{z\}\cup\{a_w:|w|<M\}\cup\{b_w:|w|<M\}.
\]
Thus $|A|=|B|=|C|=k$.  We now put a linear order on
\[
   \Omega=\{z\}\cup\{a_w,b_w:w\in\calT_M\}
\]
as follows.  Place $z$ first.  Then, for each non-empty word $u\in\calT_M$ in the order of $\calT_M$, place $a_w$ if $u=w0$ and place $b_w$ if $u=w1$.  Finally append the remaining symbols $a_w$ with $|w|=M$ in the order of $\calT_M$, followed by the remaining symbols $b_w$ with $|w|=M$ in the order of $\calT_M$.

With this order, the induced order on each of $A$, $B$ and $C$ is indexed by $\calT_M$ in the same level-lexicographic order.  The sets $A$ and $B$ are disjoint, so $\mu_\alpha(A,B)=0$.  The common elements of $A$ and $C$ are precisely the $a_w$ with $|w|<M$; in $C$, the element $a_w$ occupies the position of the child $w0$.  Therefore
\[
   \mu_\alpha(A,C)
   =\sum_{|w|<M}\alpha_w\alpha_{w0}
   =s\sum_{|w|<M}\alpha_w^2
   =s\lambda^2\sum_{j=0}^{M-1}\eta^j.
\]
Using the value of $\lambda$ and \eqref{eq:fM-choice}, this becomes
\[
   \frac{p}{\sqrt{p^2+q^2}}\sqrt{\eta}\,
   L\frac{1-\eta^M}{1-\eta^{M+1}}
   =\frac{p}{\sqrt{p^2+q^2}}\,L\rho=p.
\]
Similarly, the common elements of $B$ and $C$ are the $b_w$ with $|w|<M$, and in $C$ the element $b_w$ occupies the position of the child $w1$.  Hence
\[
   \mu_\alpha(B,C)
   =\sum_{|w|<M}\alpha_w\alpha_{w1}
   =t\sum_{|w|<M}\alpha_w^2=q.
\]
The lemma follows.
\end{proof}

\section{The angular condition}

The following elementary lemma translates the geometric angle condition into the disk inequality needed in Lemma~\ref{lem:tree}.

\begin{lemma}\label{lem:angle-disk}
Let $T$ be a non-degenerate triangle with side lengths $a,b,c$, where $c=\diam(T)$, and let $\gamma$ be the angle opposite the side of length $c$.  Put
\[
   L=\frac{c^2}{2},
   \qquad
   p=\frac{c^2-b^2}{2},
   \qquad
   q=\frac{c^2-a^2}{2}.
\]
Then
\[
   \gamma<135^\circ
   \quad\Longleftrightarrow\quad
   p^2+q^2<L^2.
\]
Moreover, $\gamma=135^\circ$ is equivalent to $p^2+q^2=L^2$, and $\gamma>135^\circ$ is equivalent to $p^2+q^2>L^2$.
\end{lemma}

\begin{proof}
Define
\[
   x=\frac{p}{L}=1-\frac{b^2}{c^2},
   \qquad
   y=\frac{q}{L}=1-\frac{a^2}{c^2}.
\]
Since $c$ is the largest side and $T$ is non-degenerate, we have $0\le x,y<1$.  By the cosine rule,
\begin{equation}\label{eq:cos-gamma}
   \cos\gamma
   =\frac{a^2+b^2-c^2}{2ab}
   =\frac{1-x-y}{2\sqrt{(1-x)(1-y)}}.
\end{equation}
If $x+y\le1$, then $\cos\gamma\ge0$, so $\gamma\le90^\circ<135^\circ$.  In this case $x^2+y^2<1$, because $x,y<1$ and $x+y\le1$.

It remains to consider $x+y>1$.  Then \eqref{eq:cos-gamma} is negative.  The inequality $\gamma\le135^\circ$ is equivalent to
\[
   \cos\gamma\ge -\frac1{\sqrt2},
\]
and, since both sides are non-positive after moving the sign, this is equivalent to
\[
   (x+y-1)^2\le 2(1-x)(1-y).
\]
Expanding and simplifying gives
\[
   x^2+y^2\le1.
\]
The strict and reversed inequalities give the strict cases.  Finally, $p^2+q^2<L^2$ is the same as $x^2+y^2<1$, and similarly for equality and the reversed inequality.
\end{proof}

\section{Proof of the main theorem}

\begin{proof}[Proof of Theorem~\ref{thm:main}]
Let $T$ have side lengths $a,b,c$, where $c=\diam(T)$, and let $\gamma$ be the angle opposite $c$.  By assumption, $\gamma<135^\circ$.  Define $L,p,q$ as in Lemma~\ref{lem:angle-disk}.  Then
\[
   p^2+q^2<L^2.
\]
By Lemma~\ref{lem:tree}, there exist $k$, non-negative coefficients $\alpha_1,\ldots,\alpha_k$ with
\[
   \sum_{j=1}^k\alpha_j^2=L,
\]
a finite linearly ordered set $\Omega$, and three $k$-subsets $A,B,C\in\binom{\Omega}{k}$ such that
\[
   \mu_\alpha(A,B)=0,
   \qquad
   \mu_\alpha(A,C)=p,
   \qquad
   \mu_\alpha(B,C)=q.
\]
Let $m=|\Omega|$.

Fix an integer $r\ge2$ and choose $n\ge R_k(m;r)$.  Form the weighted $k$-subset set
\[
   X_\alpha(n)=\{x_I:I\in\binom{[n]}{k}\}\subset\R^n
\]
as in Lemma~\ref{lem:model}.  Since the coefficients have squared sum $L$, Lemma~\ref{lem:model} gives
\[
   \diam(X_\alpha(n))=\sqrt{2L}=c=\diam(T),
\]
where we use $n\ge m\ge2k$.

Now color the points of $X_\alpha(n)$ with $r$ colors.  This induces an $r$-coloring of $\binom{[n]}{k}$ by giving $I$ the color of the point $x_I$.  This assignment is well-defined even if the map $I\mapsto x_I$ is not injective, because each point of $X_\alpha(n)$ has a single color.  By the definition of $R_k(m;r)$, there is an $m$-element set $S\subset[n]$ all of whose $k$-subsets have the same induced color.  Let
\[
   \varphi:\Omega\longrightarrow S
\]
be the order-preserving bijection.  Then the three $k$-subsets $\varphi(A)$, $\varphi(B)$ and $\varphi(C)$ have the same induced color, so the points
\[
   x_{\varphi(A)},\qquad x_{\varphi(B)},\qquad x_{\varphi(C)}
\]
are monochromatic in $X_\alpha(n)$.

By Lemma~\ref{lem:model} and the definition of the overlap form,
\[
   \langle x_{\varphi(A)},x_{\varphi(B)}\rangle=0,
   \qquad
   \langle x_{\varphi(A)},x_{\varphi(C)}\rangle=p,
   \qquad
   \langle x_{\varphi(B)},x_{\varphi(C)}\rangle=q.
\]
Since each of the three vectors has squared norm $L$, their pairwise squared distances are
\[
   \|x_{\varphi(A)}-x_{\varphi(B)}\|^2=2L=c^2,
\]
\[
   \|x_{\varphi(A)}-x_{\varphi(C)}\|^2=2(L-p)=b^2,
\]
and
\[
   \|x_{\varphi(B)}-x_{\varphi(C)}\|^2=2(L-q)=a^2.
\]
Thus these three monochromatic points form a triangle congruent to $T$.  Since the number of colors $r$ was arbitrary and the witness has diameter $\diam(T)$, the triangle $T$ is diameter-Ramsey.
\end{proof}

\section{Consequences and the endpoint}

Theorem~\ref{thm:main} and the theorem of Corsten and Frankl give the following classification for strict inequalities around $135^\circ$.

\begin{corollary}\label{cor:threshold}
Let $T$ be a non-degenerate triangle with largest angle $\gamma$.
\begin{enumerate}[label=\textup{(\roman*)}]
\item If $\gamma<135^\circ$, then $T$ is diameter-Ramsey.
\item If $\gamma>135^\circ$, then $T$ is not diameter-Ramsey.
\end{enumerate}
\end{corollary}

\begin{proof}
Part \textup{(i)} is Theorem~\ref{thm:main}.  For part \textup{(ii)}, let $c=\diam(T)$.  The circumradius of $T$ is
\[
   \crad(T)=\frac{c}{2\sin\gamma}.
\]
If $\gamma>135^\circ$, then $\sin\gamma<1/\sqrt2$, and therefore $\crad(T)>c/\sqrt2$.  After scaling $T$ to diameter $1$, Corsten and Frankl's obstruction theorem \cite{CorstenFrankl2018} implies that $T$ is not diameter-Ramsey.
\end{proof}

\begin{corollary}\label{cor:obtuse-and-conjecture}
There exist obtuse diameter-Ramsey triangles.  More precisely, every triangle whose largest angle lies in $(90^\circ,135^\circ)$ is diameter-Ramsey.  Consequently, the simplex conjecture of Corsten and Frankl \cite[Conjecture 3.1]{CorstenFrankl2018} is false.
\end{corollary}

\begin{proof}
The first assertion is immediate from Theorem~\ref{thm:main}.  For the second, take any triangle with largest angle in $(90^\circ,135^\circ)$.  Its circumcenter lies outside its convex hull, as for every obtuse triangle, but it is diameter-Ramsey by Theorem~\ref{thm:main}.  This contradicts the proposed characterization of diameter-Ramsey simplices by containment of the circumcenter in the convex hull.
\end{proof}

\begin{remark}[Endpoint]
The exact case $\gamma=135^\circ$ is not settled here.  In the notation of Lemma~\ref{lem:angle-disk}, it corresponds to the boundary $p^2+q^2=L^2$.  The finite binary-tree representation in Lemma~\ref{lem:tree} uses the strict inequality $p^2+q^2<L^2$; at the boundary the finite construction degenerates into a limiting infinite self-similar representation.  Corsten and Frankl's obstruction is also strict, since it assumes circumradius larger than $\diam/\sqrt2$.  Thus the endpoint remains a separate boundary problem.
\end{remark}

\section*{Acknowledgements}

The author is supported by the National Science Foundation of China (Nos.\ 12471329 and 12061059).

\section*{Declaration of competing interest}

The author declares no competing interests.

\section*{Data availability}

No data were used in this work.

\end{document}